\documentclass[12pt]{amsart}

\usepackage{amstext,amssymb,amsmath,stmaryrd,amsthm}
\usepackage{geometry} 
\usepackage{color}
\usepackage{hyperref}
\makeindex

\numberwithin{equation}{section}

\newtheorem{theorem}{Theorem}[section]
\newtheorem{cor}[theorem]{Corollary}
\newtheorem{Lemma}[theorem]{Lemma}
\newtheorem{proposition}[theorem]{Proposition}

\usepackage{amssymb}
\usepackage{amsmath}
\usepackage{color}
\usepackage{graphicx}
\usepackage{comment}
\usepackage{float}
\usepackage{hyperref}

\def\I{\textbf{I}}

\newcommand{\R}{\mathbb{R}}

\newcommand{\E}{\mathbb{E}}

\newcommand{\Ge}{\mathcal{L}}
\newcommand{\X}{\mathcal{I}} 

\begin{document}

\title[Concentration and Poincar\'e type  inequalities for a PJMP.]{Concentration and  Poincar\'e type  inequalities for a degenerate  pure jump Markov process.}

\author{Pierre Hodara  \and  Ioannis Papageorgiou}{ Pierre Hodara$^{*}$  \and  Ioannis Papageorgiou $^{**}$  \\  IME, Universidade de Sao Paulo }

 \thanks{\textit{Address:} Neuromat, Instituto de Matematica e Estatistica,
 Universidade de Sao Paulo, 
 rua do Matao 1010,
 Cidade Universitaria, 
 Sao Paulo - SP-  Brasil - CEP 05508-090.
\\ \text{\  \   \      } 
\textit{Email:} $^*$hodarapierre@gmail.com $^{**}$ ipapageo@ime.usp.br, papyannis@yahoo.com  \\
 This article was produced as part of the activities of FAPESP  Research, Innovation and Dissemination Center for Neuromathematics (grant 2013/ 07699-0 , S.Paulo Research Foundation); This article  is supported by FAPESP grant (2016/17655-8) $^{*}$ and  (2017/15587-8)$^{**}$ }



\begin{abstract}
We study   Talagrand concentration  and  Poincar\'e type   inequalities for  unbounded pure jump Markov processes. In particular we focus on processes with degenerate jumps that depend on the past of the whole system, based on the    model introduced  by Galves and L\"ocherbach in \cite{G-L}, in order to  describe  the activity of a biological neural network.  As a result we obtain exponential rates of convergence to equilibrium. 
\end{abstract}
 
\date{}
\maketitle 
 
\section{Introduction}
 
Our  objective  is to obtain  Poincar\'e type  inequalities for the semigroup $P_t$ and theassociated  invariant measure of non bounded jump processes inspired by the model introduced in \cite{G-L} by  Galves and L\"ocherbach, in order to     describe the interactions between  brain neurons.  As a result we obtain  exponentially fast    rates of convergence to equilibrium. There are three   interesting features about this particular  jump process. The first    is that it is characterized  by degenerate jumps, since every neuron jumps to zero after it  spikes, and thus looses its memory.  The second, is that the probability of any neuron to spike depends on its current position and so from the past of the whole neural system. Thirdly, the intensity function that describes the jump behaviour of any of the non bounded neurons at any time is an unbounded function.

 For $P_t$ the associated semigroup and $\mu$ the invariant measure   we show the  Poincar\'e type inequality
 \[\frac{1}{c(t)}\mu\left( Var_{P_t}(f))\leq \mu (\Gamma(f,f) ) +\mu (F(\phi)\Gamma(f,f)\X_D  \right)\]
 where the second term is a local term for the compact set $D:=\{x \in \R_+^N: x_i \leq m , \, 1 \leq i \leq N \}$, for some $m$. Accordingly, for every function defined outside the compact set $ \{x \in \R_+^N: x_i \leq m+1 , \, 1 \leq i \leq N \}$ we obtain the stronger
 \[\mu\left( Var_{P_t}(f)\right)\leq c(t) \mu (\Gamma(f,f) )  ).\] 
Furthermore,  we obtain  a Poincar\'e inequality for the invariant measure  
\[Var_{\mu}(f)\leq c\mu \left( \Gamma (f,f)\right).\] 
Consequently, we derive  concentration properties 
 \[\mu\left( \{P_t f- \mu(f)\geq r \}\right)\leq e^{-cr  }.\]
 In addition, we show 
 further Talagrand type concentration inequalities,
 \[\mu\left(  \left\{\sum_i^N x^i <  r\right\} \right)\geq1- e^{-cr }.\]

 Before we describe the model we present the neuroscience framework of the problem.

\subsection{the neuroscience framework}

We consider a group of finitely  many   interacting neurons, say $N$ in number. Every one of these neurons $i, \, 1 \le i \le N$ is described  by the evolution of its membrane potential  $X_t^i:\R_+ \to \R_+$ at time $t \in \R_+$.  In this way, an $N$ dimensional  random process $X_t=(X_t^1,...,X_t^N)$ is defined that represents the membrane potential of  the $N$ neurons in the network. 

The membrane potential $X_t^i$ of a neuron $i$  does not describe only the neuron itself, but also the interactions between the different neurons in the network, through the spiking activity of the neuron. What is called spike,  or alternatively  action potential, is a high-amplitude and brief depolarisation of the membrane potential that occurs from time to time, and constitutes the  only  perturbations of the membrane potential that can be propagated  from one neuron to another through chemical synapses.

The frequency with which a neuron  spikes, is expressed  through the  intensity function $\phi:\R_+ \to \R_+ $. When a neuron has membrane potential $x$, then its intensity is $\phi(x)$.  

Neurons loose their memory every time they spike, in the sense that after a neuron $i$ spikes its membrane potential is set to zero, which can be understood as the resting potential. The membrane potential of the  rest of the neurons $j\neq i$ is then increased by a quantity $W_{i \to j} \geq 0$ called the synaptic weight, which represents the influence of the spiking neuron $i$ on $j$. It should be noted that the membrane potential of any of the $N$ neurons  between two consequent  jumps remains  constant.

From our discussion up to this point it should be clear that the whole dynamic of the whole interacting neural system  is interpreted exclusively by the jump times. Thus, from a purely a probabilistic point of view, this activity can be described by a simple point process. One should however bare in mind that since the spiking neuron jumps to zero these point processes are non-Markovian. For examples of Hawkes processes describing neural systems one can look at  \cite{C17}, \cite{D-L-O},  \cite{D-O}, \cite{G-L}, \cite{H-R-R} and \cite{H-L}.

An alternative view point,  instead of focusing exclusively on the jump times, is to try to model the evolution of the membrane potential that occurs  between jumps as well, when this evolution is already determined. In the case of deterministic drift between the jumps, for example, as examined in \cite{H-K-L}, the membrane potential is attracted towards en equilibrium potential exponentially fast.  In that case, the    process is a   Piecewise Deterministic Markov Processe introduced by Davis in \cite{Davis84} and \cite{Davis93}. PDMP  processes are frequentely  used in probability to   model  chemical and biological  phenomena (see for instance  \cite{C-D-M-R} and  \cite{PTW-10}, as well as  \cite{ABGKZ} for an overview).

 In the current paper we adopt a similar framework, but in our case we do not consider a drift between the jumps, but rather a    pure jump Markov process, which for convenience we will  abbreviate as PJMP.

Although here we work with a finite number of neurons, so that we can take advantage of the Markovian nature of the membrane potential,    Hawkes processes in general allow the study of  infinite neural systems, as  in \cite{G-L} or \cite{H-L}.

 On the contrary of \cite{H-K-L}, a Lyapunov-type inequality allows us to get rid of the compact state-space assumption. Due to the deterministic and degenerate nature of the jumps, the process does not have a density continuous with respect to the Lebesgue measure. We refer the reader to \cite{L17} for a study of the density of the invariant measure. Here, we make use of the lack of  drift between the jumps to work with discrete probabilities instead of density.

\subsection{the model}

Consider the intensity function $\phi:\R_+\mapsto \R_+$, which satisfies the following conditions: 
There exist strictly  positive constants $\delta,c$ such that 
\begin{equation}\label{phi2}
\phi(x)\geq \delta
\end{equation} 
and 
\begin{align}\label{ass:dphipos}\phi(x)>cx  \ \text{for} \  x\in\R_+.\end{align}
The intensity function characterizes the  Markov process $X_t = (X^{ 1 }_t, \ldots , X^{ N}_t )$. If we define
\begin{equation}\label{eq:delta}
(\Delta_i (x))_j =    \left\{
\begin{array}{ll}
x^{j} +W_{i \to j}  & j \neq i \\
0 & j = i 
\end{array}
\right\},
 \end{equation}
 then the  generator $\Ge$ of the process $X$ is expressed through the intensity function,   by
\begin{equation}\label{eq:generator0}
\Ge f (x ) = \sum_{ i = 1 }^N \phi (x^i) \left[ f ( \Delta_i ( x)  ) - f(x) \right]
\end{equation}
for every $x \in \R_+^N$ and   $ f : \R_+^N \to \R $ any test function. Furthermore, for every    $i=1,\dots,N$ and $t\geq 0$, the Markov Process $X$ solves the following stochastic differential equation 
\begin{eqnarray}\label{eq:dyn}
X^{ i}_t &= & X^{i}_0  -  \int_0^t \int_0^\infty 
X^{ i}_{s-}  1_{ \{ z \le  \phi ( X^{ i}_{s-}) \}} N^i (ds, dz) \\
&&+    \sum_{ j \neq i } W_{j \to i} \int_0^t\int_0^\infty  1_{ \{ z \le \phi ( X^{j}_{s-}) \}} N^j (ds, dz),
\nonumber
\end{eqnarray}  
where    $(N^i(ds, dz))_{i=1,\dots,N}$ is a family of \textit{i.i.d.}\ Poisson random measures on $\R_+ \times \R_+ $ with  intensity measure $ds dz$, for some $N > 1 $ fixed.

\subsection{Poincar\'e type inequalities}
We have defined a   PJMP that describes our neural system, with  dynamics   similar to the model introduced in \cite{G-L}.  We aim  in studying      Poincar\'e type inequalities both for the semigroup $P_t$ and the invariant measure $\mu$ of the process.

  We start with a description of the analytical framework and the definition of the Poincar\'e inequality on a general discreet setting. For more details one can consult \cite{A-L}, \cite{Chaf}, \cite{D-SC}, \cite{SC} and  \cite{W-Y}. Throughout the paper we will conveniently write $\int fdv$ for the expectation of the function $f$ with respect to the measure $v$, $\int fdv$.
  
     Consider a Markov semigroup $P_t f(x)=\E^x(f(X_t))$ and the infinitesimal generator $\Ge f:=\lim_{t\rightarrow 0 +}\frac{P_t f-f}{t}$ of  a Markov process $(X_t)_{t\geq 0}$.  We will frequentely use the  following   relationships: $\frac{d}{dt}P_t= \Ge P_t=P_t \Ge$ (see for instance \cite{G-Z}).

   Furthermore, we will say that a measure $\mu$ is invariant for the semigroup $(P_s)_{s\geq 0}$ if $\mu$ satisfies 
   \[\mu P_s =\mu, \  \  \text{for every} \ \  s\geq 0.\] From the definition of the generator we obtain  that $\mu (\Ge f)=0$. 
   
   Define the "carr\'e du champ" operator $\Gamma (,)$ by 
   \[\Gamma (f,g):=\frac{1}{2}(\Ge (fg)-f\Ge g- g\Ge f).\]
In the special case of the PJMP where the infinitesimal generator $\Ge$ has the form    (\ref{eq:generator0}),    a simple calculation shows that the carr\'e  du champ has the following expression
 \begin{align*}\Gamma (f,f)= \frac{1}{2}(\sum_{ i = 1 }^N \phi (x^i) \left[ f ( \Delta_i ( x)  )- f(x) \right]^2).
 \end{align*} 
 We recall the definition of the variance of a function $f$ with respect to a probability measure $m$: $Var_{m}(f)=m(f-m(f))^2$. Having defined all the necessary increments, we can present the  definition of the classical Poincar\'e inequality.   A probability  measure $m$   satisfies the  Poincar\'e inequality  if 
 \[  Var_{m}(f)\leq Cm (\Gamma(f,f))\]
  for some strictly positive constant $C$ independent of the function $f$.  In the case where instead of a single measure we have a familly of measures as in the case of the  semigroup $\{P_t, t\geq 0\}$, then the constant $C$ may  depend on the time $t$, i.e. $C=C(t)$, as   is the case for the examples studied in \cite{W-Y}, \cite{A-L}  and \cite{Chaf}. The aforementioned papers used the so-called semigroup method that will be also followed in the current work. The nature of this method usually leads to an inequality for the semigroup $P_t$ which involves a time constant $C(t)$. 
  
  In both \cite{W-Y} and \cite{A-L}, in order to retrieve the carr\'e du champ the translation property 
  \[\E^{x+y}f(z)=\E^{x}f(z+y)\]
 was used. Taking advantage of this, for example in  \cite{W-Y},  the inequality was obtained  for a constant $C(t)=t$ for a path space of  Poisson point processes. Although this property does not hold in the degenerate PJMP examined here, we can still show that a Poincar\'e inequality, which also involves the invariant measure, holds for the semigroup $\{P_t,t\geq 0\}$ but with a time constant $C(t)$ of higher order than one.
 
In a recent paper \cite{H-P}  the same degenerate PJMP as in  (\ref{phi2})-(\ref{eq:generator0}) was considered but for bounded neurons, with membrane potential taking values in a compact set $D$
\begin{equation}\label{eq:defcompactD}
D:=\{x \in \R_+^N: x_i \leq m , \, 1 \leq i \leq N \}
\end{equation}
for some positive constant $m$.   The Poincar\'e type inequality obtained for the compact case was   
\begin{align*}Var_{ P_{t}}( f (x)) \leq  \alpha(t)P_t\Gamma(f,f)(x)  +\beta \int _0^tP_s\Gamma(f,f)(x)ds
\end{align*}
 where $\alpha(t)$ is a second order polynomial of time $t$ and $\beta$ some positive constant.  In the   more general non compact case examined in the current paper we will prove an alternative weighted Poincar\'e type inequality, which is formulated by taking the expectation with respect to the invariant measure $\mu$ for  the semigroup $(P_t)_{t\geq 0}$ in the typical Poincar\'e inequality, that is
\begin{align*}\int Var_{ P_t}( f) (x)\mu (dx) \leq  \delta_{1}(t) \mu (\Gamma(f,f)   (x)  )+\nonumber \delta_{2}(t)\mu\left( F(t,\phi) P_t ( \Gamma(f,f)\X_{x_t\in D})  (x)\right), \end{align*}
where  on the right hand  side we have also added two  local term for the compact set $D$ as in (\ref{eq:defcompactD}), the one of which   has a  weight that depends on the intensity function $\phi$.  Consequently, the stronger 
\begin{align*}\int Var_{ P_t}( f) (x)\mu (dx) \leq  \delta_{1}(t) \mu (\Gamma(f,f)   (x)  ) , \end{align*} holds for every function $f$ with a domain outside the compact   $\{x \in \R_+^N: x_i \leq m+\max_{i=1}^NW_{i \to j}, \, 1 \leq i \leq N \}$.
The reasons why in the unbounded case we focus on this particular Poincar\'e type inequality rather than the classical one about $P_t$ presented above relate with the special features that characterise the behaviour of the PJMP  process examined in the current paper. Some of them are similar to the compact case, like the    memoryless   behaviour of the neuron that spikes and as already mentioned the lack of the translation property. In the non compact case however, we also have to deal with the hindrance of controlling the  intensity function $\phi$ which is non bounded.  In order to handle the intensity   functions we will  use the Lyapunov method presented in  \cite{C-G-W-W} and \cite{B-C-G}  which has the advantage  of reducing the problem from the unbounded  to the compact case where variables take value  within the compact set  $D$ defined in (\ref{eq:defcompactD}),  that satisfies  
\[
D \supset \left\{ \sum_{i=1}^N x_i \leq m \right\},
\]
where the set $ \left\{ \sum_{i=1}^N x_i \leq m \right\}$ is the set  involved in the Lyapunov method. Since the jump behaviour depends on the current position of a  neuron this has the   benefit of   bounding the values of $\phi$ and thus controlling the spike behaviour of the neurons. 

The Lyapunov method however, as we will see later in more detail in the proof of Proposition \ref{mainProp}, requires the control of a Lyapunov function $V,$ more specifically of  $\frac{-\Ge V}{V}.$

As it will be explained in more detail later, this is a problem that although can be solved relatively easy in the case of diffusions by choosing appropriate exponential densities, in the case of jump processes it is more difficult and requires the use of invariant measures.

%
%

The inequality for the semigroup famille $\{P_t,t\geq 0\}$ which  refers to the general case where neurons take values in the whole of  $\mathbb{R}_+$ follows.
  \begin{theorem}  \label{theorem1}Assume the PJMP as described in (\ref{phi2})-(\ref{eq:generator0}). Then, for every $t\geq t_1$, for some $t_1>0,$ the following weighted Poincar\'e type inequality holds
\begin{align*}\int Var_{ E^x}( f (x_t)) \mu (dx)\leq & \delta_1(t) \mu (\Gamma(f,f)   (x)  ) + \delta_2(t)\mu\left( F(t,\phi)P_t ( \Gamma(f,f)\X_{x_t\in D})  (x)\right)
\end{align*}
where 
\[F(t,\phi)=\left(\int_0^tP_s\left( \sum_{i=1}^N\phi(x_s^i)\right)(x)ds\right),\]
while $\delta_1(t) $ a third    and $ \delta_2(t)$ a second  order polynomial  of $t$ respectively, that do not depend on the function $f$, where the set $D$ is as in (\ref{eq:defcompactD}).
\end{theorem}

As a direct corollary of the theorem we obtain the following.
\begin{cor}Assume the PJMP as described in (\ref{phi2})-(\ref{eq:generator0}). Then, for every function$f$ with a domain outside   
$ \{x \in \R_+^N: x_i \leq m+\max_{i=1}^NW_{i \to j} , \, 1 \leq i \leq N \}$, the following   Poincar\'e type inequality holds  
 \[\mu\left( Var_{P_t}(f)\right)\leq c(t) \mu (\Gamma(f,f) )  )\] 
 for every
 $t\geq t_1$, for some $t_1>0$.
\end{cor}
We conclude this section with the Poincar\'e ineqality for the invariant measure $\mu$ presented on the next theorem.
\begin{theorem} \label{invPoin} Assume the PJMP as described in (\ref{phi2})-(\ref{eq:generator0}). Then  $\mu$ satisfies a Poincar\'e inequality 
  \[\mu \left(f- \mu f \right)^2\leq C_0 \mu(\Gamma(f,f))\]
  for some constant $C_0>0$.
  \end{theorem}

   \subsection{Concentration and other Talagrand type inequalities}
Concentration inequalities play a vital role in the examination of a system's convergence to equilibrium. Talagrand  (see \cite{Ta1}  and \cite{Ta2})
associated  the log-Sobolev and Poincar\'e inequalities  for exponential distributions with concentration properties (see also \cite{Bo-Le}), that is
\begin{align}\label{genCon}\mu \left(P_t f- \mu (f) >r \right)\leq \lambda_0  e^{\lambda \mu (F)} e^{-\lambda  r^p}\end{align}
for some $p\geq 1$.
In particular, when the log-Sobolev  inequality holds, then (\ref{genCon}) is true for $p=2$, while in the case of the weaker Poincar\'e inequality, the exponent is $p=1$. Furthermore, the modified log-Sobolev inequality that interpolates between the two, investigated for example in \cite{B-R}, \cite{G-G-M} and \cite{Pa},  gives  convergence to equilibrium of speed $1<p<2$.

The problem of  concentration properties  for measures that satisfy a Poincar\'e inequality, or as in our case, the Poincar\'e type inequality, is closely related with exponential integrability of the measure, that is
\[\mu (e^{\lambda f})<+\infty\]
for some appropriate class of functions $f$. This problem, is itself connected to bounding the carr\'e  du champ of the exponent of a function 
\begin{align}\label{AidaBound}\mu(\Gamma(e^{\lambda f/2},e^{\lambda f/2}))\leq \frac{ \lambda^2}{4} \Psi (f) \mu(e^{\lambda f})  \end{align} for some $\Psi (f)$ uniformly bounded. In the case of diffusion processes where the carr\' e du champ is defined through a derivation, (\ref{AidaBound}) is satisfied for $\vert \vert \nabla f \vert \vert_{\infty}<1$ (see section   \ref{proofTal} for more details).  For a detailed discussion on the subject one can look at \cite{Led0}.   In our case we consider 
\[\vert\vert\vert f\vert\vert\vert_{\infty}=\sup   \left\{\mu(fg); \  g: \mu(g)\leq 1\right\} .\]
Then we can obtain  exponential integrability and a bound (\ref{AidaBound}) for  functions $f$ such that 
 $\vert \vert \vert \phi(x^i)   D(f)^2\vert \vert \vert_{\infty}< 1$ and   $\vert \vert \vert \phi(x^i) e^{\lambda D(f)} D(f)^2\vert \vert \vert_{\infty}< 1$. This, together with the Poincar\'e type inequality already obtained, can show concentration properties for a different class    of functions than the ones assumed in Corollary \ref{thmCon1},  
as presented in the next theorem.
\begin{theorem}\label{thmCon2} Assume the PJMP as described  in (\ref{phi2})-(\ref{eq:generator0}).  For every  function $f$, such that $\mu(f)<\infty$, satisfying 
\[\vert \vert \vert \phi(x^i)   D(f)^2\vert \vert \vert_{\infty}< 1\text{\ \ and \  \ }\vert \vert \vert \phi(x^i) e^{\lambda D(f)} D(f)^2\vert \vert \vert_{\infty}< 1\]  there exists a constant $\lambda_0>0$ such that
\[\mu \left( \left\{f- \mu (f) >r\right\} \right)\leq \lambda_0  e^{\lambda \mu (f)} e^{-\lambda  r},\]
for some $\lambda,\lambda_{0}>0$. 
\end{theorem}
Consequently   we obtain the following convergence to equilibrium property:
\begin{cor}\label{thmCon1} Assume $\mu (f)<\infty$. Assume the PJMP as described  in (\ref{phi2})-(\ref{eq:generator0}).  For every  function $f$, satisfying 
\[\vert \vert \vert \phi(x^i)   D(f)^2\vert \vert \vert_{\infty}< 1\text{\ \ and \  \ }\vert \vert \vert \phi(x^i) e^{\lambda D(f)} D(f)^2\vert \vert \vert_{\infty}< 1\]  there exist  constants $\lambda, \lambda_0$ such that
\[\mu \left(\left\{P_t f- \mu (f) >r\right\} \right)\leq \lambda_0  e^{\lambda \mu (f)} e^{-\lambda  r}\] 
where $\mu$ is the invariant measure of the semigroup $P_t$.
\end{cor}

Furthermore, for the case of unbounded neurons, we can obtain Talagrand inequalities in the spirit of the ones proven for the modified log-Sobolev in  \cite{B-R}.
 \begin{theorem}  \label{talagTheo}Assume the PJMP as described in (\ref{eq:generator0})-(\ref{phi2}). Then the following Talagrand  inequality holds.
\[\mu\left(\left\{  x:\sum_i x^i<r  \right\}\right)\geq 1- \lambda_0 e^{-\lambda r}\]
for some $\lambda_{0},\lambda>0$.\end{theorem}

A few words about the structure of the paper.   The proof of the  Poincar\'e inequality for the   semigroup $P_t$ and the invariant measure $\mu$ are presented in sections \ref{PoinPt} and \ref{PoinInvariant} respectively. For both inequalities a Lyapunov inequality will be used to control the behaviour of the neurons outside a compact set. This is proven at the begining of section  \ref{main3}.
 In the final section \ref{proofTal} the concentration inequalities are proven. At first in Proposition \ref{nnpropcon} we present the main tool that connects the Poincar\'e type inequality with the concentration properties. Then,   the required conditions are verified for the PJMP.

 \section{proof of the Poincar\'e inequalityies}\label{main3}
 In both the inequalities involving the semigroup and the invariant measure, the use of a  Lyapunov function will be a crucial tool in order to control the intensity function outside a compact set. 
   
 At first we will work towards deriving the Lyapunov inequality required. That will be the subject of the next lemma.  
 
We recall that under the framework of \cite{H-K-L}, the generator of our process is given, for any function $f,$ by
$$
 \Ge f(x) = \sum_{i=1}^N \phi (x^i) \left( f \left( \Delta_i(x) \right) - f(x) \right) 
$$

where  $\Delta_i(x)$ is defined by $\left( \Delta_i(x) \right)_j := x_j + W_{i \to j}$ if $j \neq i$ and  $\left( \Delta_i(x) \right)^i := 0.$

We assume that for all $i, j, W_{i \to j} \geq 0,$ we can then consider that the state space is $\R_+^N.$

We put $W_i := \sum_{j \neq i} W_{i \to j}$.

\begin{Lemma}\label{LyapLem} Assume that for all $x \in \R_+, \,  \phi(x) \geq cx$ and $  \delta \leq \phi(x)$ for some constants $c$ and $\delta > 0$. Then if we consider the   Lyapunov function:
\[
V(x)=1+\sum_{i = 1}^N x_i.
\]
there exist positive constants $\vartheta, \, b$ and $m$ so that the following Lyapunov inequality holds 
\begin{align*}\Ge V\leq -\vartheta V + b\X_B
\end{align*}
for  the set $
B= \left\{ \sum _{i=1}^N x^i \leq m \right\}$.
\end{Lemma}
\begin{proof}
For the Lyapunov function $V$ as stated before, we have 
\begin{align*}
\Ge V(x) & = \sum _{i=1}^N \phi(x^i) (W_i - x^i) \\
& =\sum_{i:x^i>1+W_i} \phi(x^i) (W_i - x^i) + \sum_{i:x^i \leq 1+W_i} \phi(x^i) (W_i - x^i) \\
& \leq - \sum_{i:x^i>1+W_i} \phi(x^i) + \sum_{i:x^i \leq 1+W_i} \phi(1 + W_i) W_i - \delta \sum_{i:x^i \leq 1+W_i} x^i \\ 
& \leq - (c \wedge \delta )\sum _{i=1}^N x^i + \sum _{i=1}^N\phi(1 + W_i) W_i.
\end{align*}
Putting $ b:= \sum _{i=1}^N\phi(1 + W_i) W_i,$ we have for any $\alpha \in [0,1]$
\begin{align*}
\Ge V(x) & \leq - \alpha (c \wedge \delta ) \sum _{i=1}^N x^i - (1- \alpha ) (c \wedge \delta )\sum _{i=1}^N x^i  +b \\
& \leq - \alpha (c \wedge \delta ) V(x) +b+ \alpha (c \wedge \delta )  - (1- \alpha ) (c \wedge \delta )\sum _{i=1}^N x^i \\
& \leq - \alpha (c \wedge \delta ) V(x) + \left( b+ \alpha (c \wedge \delta ) \right) 1_B(x),
\end{align*}
with 

$$
B= \left\{ \sum _{i=1}^N x^i \leq \frac{b+ \alpha (c \wedge \delta )}{(1- \alpha ) (c \wedge \delta )} \right\}
$$

in which case $m=\frac{b+ \alpha (c \wedge \delta )}{(1- \alpha ) (c \wedge \delta )}$. Since $\alpha$ can be chosen arbitrary close to 1, if we want to impose $\alpha (c \wedge \delta ) >1,$ we need to assume that $c>1$ and $\delta >1.$
\end{proof}

In the following subsection we show the weighted  Poincar\'e for the semigroup $P_t$, while in subsection  \ref{PoinInvariant} we show the inequality for the invariant measure.

  \subsection{Poincar\'e inequality for the semigroup}\label{PoinPt}
  
  ˜

In this section we prove  the main results of the paper for  systems of neurons that take values on $\mathbb{R}_+$, presented in  Theorem \ref{theorem1}. 

As mentioned in the introduction, the approach used will be to reduce the problem from the unbounded case to the compact case examined in  \cite{H-P}. To do this we will  follow  closely the Lyapunov approach developed in \cite{C-G-W-W} and \cite{B-C-G} to prove superPoincar\'e inequalities.

We start by showing that the chain returns  to the compact set $D$ with a strictly  positive probability bounded from below.

For a   neuron $i \in I$ and time $s$, we define $p_s(x)$ to be the probability that the process starting with initial configuration $x$ has no jump during time $s,$ and $p_s^i(x)$ the probability that the process has exactly one jump of neuron $i$ and no jumps for other neurons during time $s$. Then,
$$
p_s(x)= e^{-s \overline{\phi}(x)}
$$
and
\begin{multline}\label{computePsix}
p_s^i(x)= \int_0^s \phi(x_i) e^{-u\overline{\phi}(x)}e^{-(s-u)\overline{\phi}(\Delta^i(x))}du 
\\ = \left\{
\begin{array}{ll}
\frac{\phi(x_i)}{\overline{\phi}(x) - \overline{\phi}(\Delta^i(x))} \left( e^{-s\overline{\phi}(\Delta^i(x))} - e^{-s\overline{\phi}(x) } \right) & \, \mbox{if } \,  \overline{\phi}(\Delta^i(x)) \neq \overline{\phi}(x) \\
s \phi(x_i) e^{-s\overline{\phi}(x) } & \, \mbox{if } \, \overline{\phi}(\Delta^i(x)) = \overline{\phi}(x) 
\end{array}
\right\} 
\end{multline}
where above  we have denoted $\overline{\phi}(x)=\sum_{j \in I} \phi(x_j)$. Furthermore, if we denote  
\begin{equation}\label{def:sm}
t_0=  \left\{
\begin{array}{ll}
\frac{ln\left( \overline{\phi}(x) \right) -ln \left( \overline{\phi}(\Delta^i(x)) \right)}{\overline{\phi}(x) - \overline{\phi}(\Delta^i(x))} & \, \mbox{if } \,  \overline{\phi}(\Delta^i(x)) \neq \overline{\phi}(x) \\
\frac{1}{\overline{\phi}(x)} & \, \mbox{if } \, \overline{\phi}(\Delta^i(x)) = \overline{\phi}(x) 
\end{array}
\right\} 
\end{equation}
then 
  $ p_s^i(x)$ as a function of the time $s$, is continuous, strictly increasing on $(0, t_0)$ and strictly decreasing on $(t_0,+\infty)$, while  we have $p_0^i(x)=0$.

For any configuration $y \in D$ we define the set of configurations $D_y$ containing all configurations $x$ such that for some $t>0, \, \pi_t(x,y):=P_x(X_t=y) >0.$

 \begin{Lemma}\label{boundProb} 
 Assume the PJMP as described in (\ref{phi2})-(\ref{eq:generator0}). Then, for every $y\in D$ and  $x\in D_y$, 
 \[\pi_t(x,y)\geq \frac{1}{\theta}\] for $t\geq t_1=\frac{1}{\delta}+t_0$. \end{Lemma}
 
 \begin{proof} 
 We want to show   that for every  configuration $y\in D$ that belongs to the domain of the invariant measure, one has that $\pi_t(x,y)\geq \frac{1}{\theta}$ for some positive $\theta$.  The proof will be divided in three parts.

A)  At fist, for $y\in D$, we restrict ourselves to every $x\in D \cap D_y$.

 Since $\mu (y)>0$ and $\lim_{t\rightarrow \infty }\pi_t(x,y)=\mu (y)$    we readily obtain that for every couple $x,y\in D$ there exist    $\theta_1>0$ and   $t_{x,y}>0$ such that for every $t>t_{x,y}$ we have that   $\pi_t(x,y)>\frac{1}{\theta_1} $. But since $D$ is compact, the configurations  in $D$ are finite in number and so $\max_{x,y\in D}\{t_{x,y}\}<\infty$. We thus conclude that there exists a $\theta_1>0$ such that 
 \[\pi_t(x,y)>\frac{1}{\theta_1} \]
 for every $t>\max_{x,y\in D}\{t_{x,y}\}$.

 In the next two steps  we extend the last result to $x\in D^c$.

 B) We will show that there exist $\theta_2>0$ and $\frac{1}{\delta}>t_2>0$, such that  for every $x\in D^c\cap D_y$ there exists a $z\in D\cap D_y$   such that 
 \[\pi_{t_2}(x,z)>\frac{1}{\theta_2}. \]

We enumerate the $N$ neurons with numbers from $1$ to $N$ on decreasing order, so that $\phi(x_i)\geq \phi(x_{i+1})$. Define $\hat x^i=\Delta_i(\Delta_{i-1}(...\Delta_1(x))...)$ the configuration starting from $x$ after the $1st$, then the $2nd$ up to the time the $i$'th neuron has spiked in that order. Then for every $s_i>0$  we have 
 \[\pi_{t_3}(x,\hat x^N)\geq p_{s_1}^1(x)p_{s_2}^2(\hat x^1)...p_{s_N}^N(\hat x^{N-1})\] 
 where we recall that $p_s^i(x)$ is the probability that the process starting from $x$  has exactly one jump of the neuron $i$ in time $s$ and no jumps of other neurons.  If we choose $s_i=\frac{1}{N\phi (\hat x^{i-1}_i)}$ then we have $p_{s_i}^i(\hat x^{i-1})  \geq  
N^{-1}e^{-1 }$. To see this,  from (\ref{computePsix}) we can  compute bounds for $p_{s_i}^i(\hat x^{i-1})$. In the case where $ \overline{\phi}(\Delta^i(\hat x^{i-1})) = \overline{\phi}(\hat x^{i-1}) $ we have 
\[
p_{s_i}^i(\hat x^{i-1})=
s_i \phi(\hat x^{i-1}_i) e^{-s_i\overline{\phi}(\hat x^{i-1}) }   \geq  
N^{-1}e^{-1 }
 . 
\]since $\phi(\hat x^{i-1}_i)\geq \phi(\hat x^{i-1}_j)$ for every $j\neq i$, implies that $  \frac{\overline{\phi}(\hat x^{i-1})}{N \phi(\hat x^{i-1}_j)}\leq 1$. In the opposite case where $\overline{\phi}(\Delta^i(\hat x^{i-1})) \neq \overline{\phi}(\hat x^{i-1})$, we then have
  \begin{align*} 
p_{s_i}^i(\hat x^{i-1})=
\frac{\phi(\hat x^{i-1}_i)}{\overline{\phi}(\hat x^{i-1}) - \overline{\phi}(\Delta^i(\hat x^{i-1}))} \left( e^{-s_i\overline{\phi}(\Delta^i(\hat x^{i-1}))} - e^{-s_i\overline{\phi}(\hat x^{i-1}) } \right) \\ \geq  
\frac{s_i\phi(\hat x^{i-1}_i)}{\overline{\phi}(\hat x^{i-1}) - \overline{\phi}(\Delta^i(\hat x^{i-1}))}  e^{-s_i\min \{\overline{\phi}(\Delta^i(\hat x^{i-1})),\overline{\phi}(\hat x^{i-1}) \}}\left(  \overline{\phi}(\hat x^{i-1})  -\overline{\phi}(\Delta^i(\hat x^{i-1})) \right)
\\ \geq  
\frac{1}{N}  e^{-1  }  . 
\end{align*}
since $\frac{ \overline{\phi}(\Delta^i(\hat x^{i-1}))}{N\phi (\hat x^{i-1}_i)} \leq 1$ and $\frac{ \overline{\phi}(\hat x^{i-1})}{N\phi (\hat x^{i-1}_i)} \leq 1$.  

 So we obtain
 \[\pi_{t_2}(x,z)\geq  (Ne)^{-N},\] 
and the result is  proven for $\theta _2=(Ne)^N$, $z=\hat x^N$ and $t_2\leq \sum_{i=1}^Ns_i\leq \frac{1}{\delta}$.

C) Having shown (A) and (B) we can now  complete the proof of the lemma for $x\in D^c$. For this, it  is sufficient, for every $y\in D$ and $x\in D^c\cap D_y$ to write 
 \[\pi_t(x,y)\geq \pi_{t_3}(x,\hat x^N)\pi_{t_2}(\hat x^N,y)\] 
 and the assertion follows for $t\geq \frac{1}{\delta}+t_2$. Consequently, the lemma follows for $t\geq \max\{t_1,t_2+\frac{1}{\delta}\}$.
\end{proof}

Taking under account  the last result, we can obtain the first technical bound needed in the proof of  the local Poincar\'e inequality, taking  advantage of the bounds shown for times bigger than $t_1$. 

\begin{Lemma}\label{N1techn} Assume $z\in D^c$. For  the PJMP as described in (\ref{phi2})-(\ref{eq:generator0}), we have
 \begin{align*}\left(\int_0^{t-s}\left(\E^{\Delta_i ( z) }( \Ge f(z_{u})\X_{z_u\in D})-  \E^z(\Ge f(z_{u})\X_{z_u\in D})\right)du\right)^2 \leq \\  4\theta^2 t^2 M  \E^x ( \Gamma(f,f)(x_t)\X_{x_t\in D}) 
 \end{align*}
 for every  $t\geq t_1$. 
\end{Lemma}
\begin{proof} 
We can compute
   \begin{align}\nonumber
\I_2:=&\left(\int_0^{t-s}\left(\E^{\Delta_i ( z) }( \Ge f(z_{u})\X_{z_u\in D})-  \E^z(\Ge f(z_{u})\X_{z_u\in D})\right)du\right)^2\leq \\ \nonumber      & 2 \left(\int_0^{t-s}\sum_{y\in D}\pi_u(\Delta_i (z),y) ( \Ge f(y)) du\right)^2+2 \left(\int_0^{t-s}\sum_{y\in D}\pi_u(z,y)( \Ge f(y)) du\right)^2
\end{align} 
Since $t\geq t_1$, we can use Lemma \ref{boundProb} to bound for every $w$ and $y\in D$,  $\pi_u(w,y)\leq \theta \pi_t(x,y)$ we obtain 
   \begin{align}\nonumber
\I_2\leq  &  \nonumber       4\theta^2 \left(\int_0^{t-s}\sum_{y\in D}\pi_t(x,y) ( \Ge f(y)) du\right)^2 \\   \nonumber =&4\theta^2 t^2 \left( \sum_{y\in D}\pi_t(x,y) (  \sum_{ i = 1 }^N \phi (y^i)(f(\Delta_i ( y))-f(y)))  \right)^2.
\end{align} Now we will use twice the Cauchy-Schwarz inequality, to pass the square inside the two sums. We will then obtain  
    \begin{align}\nonumber
\I_2\leq  &  \nonumber    4\theta^2 t^2 M \sum_{y\in D}\pi_t(x,y)   \sum_{ i = 1 }^N \phi (y^i)\left(f(\Delta_i ( y))-f(y)  \right)^2\\  \nonumber  &=  4\theta^2 t^2 M \E^x ( \Gamma(f,f)(x_t)\X_{x_t\in D}). 
\end{align} 
\end{proof}

\begin{Lemma}\label{lastPoin}For  the PJMP as described in (\ref{phi2})-(\ref{eq:generator0}), we have
  \begin{align*}
 \E^{x} (f^2(x_t)\X_{x_t\in D})-(\E^{x} f(x_t)\X_{x_t\in D}))^2\leq & 2\int_0^tP_s\Gamma(f,f)( x)ds +\\   +8\theta^2 t^2 M\left(\int_0^tP_s\left( \sum_{i=1}^N\phi(x_s^i)\right)(x)ds\right)&\E^{x} ( \Gamma(f,f)(x_t)\X_{x_t\in D}))(x_t)\X_{x_t\in D})
 \end{align*}
  
for every $t\geq t_1$.
\end{Lemma}
\begin{proof}

  Consider  the semigroup $P_tf(x)=\E^xf(x_t)$. Since $\frac{d}{ds}P_s= \Ge P_s =P_s \Ge$, we can calculate
 \begin{align}
 P_t f^2(x)-(P_t f(x))^2=\int_0^t \frac{d}{ds}P_s (P_{t-s}f)^2(x) ds=\int_0^t P_s \Gamma  (P_{t-s}f,P_{t-s}f)(x) ds \label{pr2.1}.
 \end{align}
 We want to bound the carr\'e du champ of the semigroup on the right hand side $ \Gamma  (P_{t-s}f,P_{t-s}f)$ by the semigroup of the carr\'e du champ $ P_{t-s}\Gamma  (f, f)$ so that the energy of the Poincar\'e inequality will be formed. 
If the process is such that the translation property $\E^{x+y}f(z)=\E^{x}f(z+y)$ holds, as in   \cite{W-Y} and \cite{A-L}, then one can obtain the desired bound as shown below.  
\begin{align*}\Gamma  (P_{t-s}f,P_{t-s}f)(x)&=\sum_{ i = 1 }^N \phi (x^i)( \E^{\Delta_i ( x) }f(x_{t-s})-\E^{x}f(x_{t-s}))^2  \\ & = \sum_{ i = 1 }^N \phi (x^i) (\E^{x }f(\Delta_i (x_{t-s}))-\E^{x}f(x_{t-s}))^2 \leq  P_{t-s}\Gamma(f,f)(x) \end{align*}
In our case where the degeneracy of the process does not allow for the translation property to take hold we will use a bound based on the Dynkin's formula. If we then use Dynkin's formula
\[\E^y f(x_t)=f(y)+\int_0^{t}\E^y(\Ge f(x_u))du\] 
we can consequently  bound  
  \begin{align*}\nonumber
\left(\E^{\Delta_i (y)}f(x_{t-s})-\E^{y}f(x_{t-s})\right)^2   \leq &2  \left(f(\Delta_i ( y) )-  f(y)\right)^2 +\\ & +2 \left(\int_0^{{t-s}}\left(\E^{\Delta_i ( y) }( \Ge f(x_{u}))-  \E^y(\Ge f(x_{u}))\right)du\right)^2. 
 \end{align*}
  In order to bound the second term above we will use the bound shown in  Lemma \ref{N1techn}  
     \begin{align*}
\left(\E^{\Delta_i ( y)}f(x_{t-s})-\E^{y}f(x_{t-s})\right)^2   \leq & 2  \left(f(\Delta_i ( y) )-  f(y)\right)^2 \\ &+8\theta^2 t^2 M  \E^x ( \Gamma(f,f)(x_t)\X_{x_t\in D}). 
 \end{align*}
 By the definition of the carr\'e du champ we then get
  \begin{align*}
\Gamma  (P_{t-s}f,P_{t-s}f)(x_s) \leq &2\Gamma(f,f)( x_s)+8\theta^2 t^2 M \left( \sum_{i=1}^N\phi(x_s^i)\right)\E^{x} ( \Gamma(f,f)(x_t)\X_{x_t\in D}) . 
 \end{align*}
If we combine the last one together with    (\ref{pr2.1}) we obtain 
 \begin{align*}
 P_t f^2(x)-(P_t f(x))^2\leq & 2\int_0^tP_s\Gamma(f,f)( x)ds +\\   &8\theta^2 t^2 M\left(\int_0^tP_s\left( \sum_{i=1}^N\phi(x_s^i)\right)(x)ds\right)\E^{x} ( \Gamma(f,f)(x_t)\X_{x_t\in D})
 \end{align*}
 \end{proof}
 From the last lemma we obtain the following local Poincar\'e inequality.
\begin{cor}\label{lastcoro1}For  the PJMP as described in (\ref{phi2})-(\ref{eq:generator0}), we have
\begin{align*}\mu (f^2\X_ D) \leq& \mu ( \left( \E^x ( f(x_t)\X_{x_t \in D})   \right)^2)+a_1(t)\mu(\Gamma((f,f)) (x)\X_{x  \in D})+
  \\  &a_2(t)\mu\left(\left(\sum_{i=1}^NP_t\phi(x_t^i)(x)\right)\E^{x} ( \Gamma(f,f)(x_t)\X_{x_t\in D})\right).
\end{align*}
where $a_1(t)= 2 t $ and $a_2(t)=8\theta^3 t^3 M$.
\end{cor}
\begin{proof}Since for  $\mu$  the invariant measure of $P_t$ one has $\mu(x)=\sum_y \mu (y)P_t(y,x)$ we can write 
\begin{align}\nonumber \mu (f^2\X_ D)=&\sum_{x\in D}\mu(x)f^2(x)=\sum_{x\in D} \sum_y \mu (y)P_t(y,x)f^2(x)=\\  = &   \sum_y \mu (y)\sum_{x\in D} P_t(y,x)f^2(x).\label{inv} \end{align}
If we now use Lemma \ref{lastPoin} to bound the semigroup we obtain
\begin{align*}\mu (f^2\X_ D) \leq& \mu ( \left( \E^x ( f(x_t)\X_{x_t \in D})   \right)^2)+2t\mu(\Gamma((f,f)) (x)\X_{x  \in D})+
  \\  &8\theta^3 t^3 M\mu\left(\left(\sum_{i=1}^NP_t\phi(x_t^i)(x)\right)\E^{x} ( \Gamma(f,f)(x_t)\X_{x_t\in D})\right).
\end{align*}
\end{proof}
Since we have already obtained local Poincar\'e inequalities, as well as  the Lyapunov inequality required,  in the following proposition  we show how the two conditions, the local Poincar\'e of Corollary \ref{lastcoro1}  and the Lyapunov inequality of Lemma \ref{LyapLem}, are sufficient for the Poincar\'e type inequality of Theorem \ref{theorem1}.

 Having obtained all the elements required, we can finish the proof of the theorem by showing how the local Poincar\'e obtained in the last corollary together with the Lyapunov inequality shown in the last lemma lead to the desired Poincar\'e type inequality.
 
\begin{proposition} \label{mainProp} Assume that for some $V\geq 1$ the Lyapunov inequality
\begin{align*}\Ge V\leq -\vartheta  V + b\X_B
\end{align*}
holds and that for  some   $D\supset B$ we have the weighted local Poincar\'e 
\begin{align}\label{newPoin}\nonumber \mu (f^2\X_ D) \leq& \mu ( \left( \E^x ( f(x_t)\X_{x_t \in D})   \right)^2)+a_1(t)\mu(\Gamma((f,f)) (x)\X_{x  \in D})+
  \\  &a_2(t)\mu\left(B_t(x)\E^{x} ( \Gamma(f,f)(x_t)\X_{x_t\in D})\right) 
\end{align}
where $B_t(x)$ a function of the semigroup $P_t$ of some function with initial configuration $x$.
  Then
\begin{align*}\int Var_{ E^x}( f (x_t))d\mu \leq & \delta(t) \mu (\Gamma(f,f)   (x)  )+ \nonumber a_2(t)\mu\left( B_t(x) P_t ( \Gamma(f,f)\X_{x_t\in D})  (x)\right) 
\end{align*}
where $\delta (t)=a_1(t)+ \frac{d_1}{2 \lambda}$, for some $d_1>0$.
\end{proposition}

\begin{proof}At first, we can write
\begin{align}\mu ( f^{2})=&\mu ( f^{2}\X_{D}) +\frac{1}{\vartheta}\mu ( f^{2}\vartheta \X_{D^c}). \nonumber\end{align}
We can bound the first   term on the right hand side  from   (\ref{newPoin}). 
 For the  second term  we can use the Lyapunov inequality. That   gives 
\begin{align*}\mu ( f^{2}\vartheta\X_{D^c}) \leq &\mu ( f^2\frac{-\Ge V}{V}\X_{D^c})+b\mu (   f^2\X_{B\cap D^{c}}).
\end{align*}
If we choose $D$ large enough to contain the set $B$, i.e.  $B\cap D^c=\emptyset$ the last one is reduced to 
\begin{align*}\mu ( f^{2}\vartheta\X_{D^c}) \leq \mu (  f^2\frac{-\Ge V}{V}\X_{D^c}).
\end{align*}
The need to bound the  quantity $\frac{-\Ge V}{V}$ which appears from  the use of the Lyapunov inequality is the actual reason why we need to make use of the invariant measure $\mu$ and obtain the type of Poincar\'e inequality shown in our final result, rather  than the  Poincar\'e type inequality based exclusively on the $P_t$ measure obtained in the previous section for the compact case. If we had not taken the expectation with respect to the  invariant measure,  we would had needed to bound
$$ \int   f^2\frac{-\Ge V}{V}\X_{D^c}dP_t$$
instead. This, in the case of diffusions can be bounded by the carr\'e  du champ of the function $\Gamma(f,f)$ by making an appropriate selection of exponential decreasing density (see for instance  \cite{B-C-G},  \cite{B-G-L} and \cite{C-G-W-W}). In  the case of jump processes however, and in  particular of PJMP  as on the current paper where densities cannot  been specified,  a similar bound cannot be obtained. However, when it comes to the analogue expression involving the invariant measure there is a powerful result that we can use, which has been presented in  \cite{C-G-W-W} (see Lemma 2.12). According to this, when the expectation is taken with respect to the invariant measure,  the desired bound holds as seen in the following lemma.
\begin{Lemma}   (\cite{C-G-W-W}: Lemma 2.12) For every $U\geq 1$ such that $-\frac{\Ge U}{U}$ is bounded from below, the following bound holds
\begin{align*}\mu f^2\frac{-\Ge U}{U} \leq  d_{1}\mu (f(-\Ge) f)  \end{align*}
where $\mu$ is the  invariant measure of the process and $d_1$ is some positive constant.
\end{Lemma}
Since $V\geq 1$ and for $x\in D$  we have  from the Lyapunov inequality that $-\frac{\Ge V}{V}\geq \vartheta$   we get the following bound
\[\mu ( f^{2}\vartheta\X_{D^c} ) \leq   d_{1}\mu  (f(-\Ge) f)    \]
for some positive constant $d_1$. Since  for the infinitesimal operator $\mu (\Ge f)=0$ for every function $f$, we can write 
$$\int (f(-\Ge) f)    d\mu=\frac{1}{2}\int (\Ge(f^2)-2f\Ge f)    d\mu=\frac{1}{2}\int \Gamma(f,f) d\mu.$$
So that,
\begin{align}
\mu ( f^{2}\lambda\X_{D^c} ) \leq   \frac{d_{1}}{2}\mu (\Gamma(f,f)    ). \label{mp3}
\end{align}

 Gathering  all  together   we finally obtain
 the desired inequality 
\begin{align*}\nonumber \int  f^{2}d\mu\leq & (a_1(t)+ \frac{d_{1}}{2\vartheta})\mu (\Gamma(f,f)   (x)  )+ a_2(t)\mu\left( B_t(x) P_t ( \Gamma(f,f)\X_{x_t\in D})  (x)\right) +\\ &+\mu ( \left( \E^x ( f(x_t)\X_{x_t \in D})   \right)^2)  \end{align*}
which proves the proposition  for  a constant $\delta (t)=a_1 (t)+ \frac{d_{1}}{2\vartheta}$.

\end{proof}

The last proposition  together with the Lyapunov inequality from Lemma \ref{LyapLem} and the local Poincar\'e inequality of Corollary \ref{lastcoro1}   proves Theorem \ref{theorem1}.

 \subsection{proof of the Poincar\'e inequalities for the invariant measure}\label{PoinInvariant}

˜

In the next proposition we see how the Lyapunov inequality is sufficient to prove a Poincar\' e inequality for the invariant measure $\mu$ presented in Theorem \ref{invPoin}, using methods developed in   \cite{B-C-G},  \cite{B-G-L} and \cite{C-G-W-W}.

\begin{proposition} \label{propCon}For  the PJMP as described in (\ref{phi2})-(\ref{eq:generator0})), assume   that for some $V\geq 1$ the Lyapunov inequality
\begin{align*}\Ge V\leq -\vartheta  V + b\X_B
\end{align*}
holds. Then  $\mu$ satisfies a Poincar\'e inequality 
  \[\mu \left(f- \mu f \right)^2\leq C_0 \mu(\Gamma(f,f)),\]
  for some $C_0>0$.
  \end{proposition}

\begin{proof}  At first assume $\mu (f\X_D)=0$. We can write
\begin{align*}Var_{\mu}(f)=\int  f^{2}d\mu-\left(\int f\X_{D^c}d_\mu\right)^2\leq &\int   f^{2}\X_{D }d\mu +\frac{1}{\vartheta}\int  f^{2}\vartheta\ \X_{D^c}d\mu.  \end{align*}
For the second   term if we work as in Proposition \ref{mainProp}, with the use   of the Lyapunov inequality  we have the following bound 
\begin{align*}\int f^{2}\vartheta\X_{D^c}d\mu \leq \int   e^{\lambda f}\frac{-\Ge V}{V}\X_{D^c}d\mu\leq d_{1}\int \Gamma(f,f)\X_{D^c}    d\mu.
\end{align*}
For the first term, we will   use  the approach applied in   \cite{SC} in order to prove  Poincar\'e inequalities  for finite Markov chains. Since we have assumed  $\int f\X_Dd\mu=0$,  we can write
\[\int   f^{2}\X_{D}d\mu=\frac{1}{2}\int \int \left( f(x)-f(y)\right)^{2}\X_{x\in D}\X_{y\in D}\mu(dx)\mu(dy).\]
If we consider $J_{xy}=\left\{J_1,...,J_{\|J_{xy}\|}\right\}$ to be the shortest sequence of spikes that leads from the configuration $x$ to the configuration $y$ without leaving $D$, then we can  denote   $\tilde x^0=x$ and for every  $k=0,...,\| J_{xy}\|$, $\tilde x^{k}=\Delta_{J_k}(\Delta_{J_{k-1}}(...\Delta_{J_1}(x))...)$, the configuration after the $k$th neuron on the sequence has spiked. Since $D$ is finite, the length of the sequence is always uniformly  bounded for any couple $x,y\in D$.  We can then write
 \[\mu(x)\mu(y) ( f(x)-f(y))^2\leq  \mu (y)\mu (x)  \sum_{j=0}^{\vert J_{xy}\vert}( f(\Delta(\tilde x^j)_{J_j})-f(\tilde x^j))^2.\]
 Since $\phi\geq \delta$ we have 
 \[ \mu(x)\mu(y) ( f(x)-f(y))^2\leq \frac{ \mu (y)\mu (x)}{ \delta}  \sum^{\vert J_{xy}\vert}_{j=0}\varphi(\tilde x^i_{J_j})( f(\Delta(\tilde x^j)_{J_j})-f(\tilde x^j))^2. \]
  If we  form the car\'e du champ, we will obtain
 \begin{align*}\mu(x)\mu(y) ( f(x)-f(y))^2\leq &\frac{ \mu(y)\mu (x)}{ \delta} \sum^{\vert J_{xy}\vert}_{j=0}\sum_{i\in D}\varphi(\tilde x^i_{J_j})( f(\Delta(\tilde x^j)_{J_j})-f(\tilde x^j))^2
\\ 
\leq & \frac{ \mu (y)\mu (x)}{ \min \{x\in D:\mu(x)\}\delta}  \sum^{\vert J_{xy}\vert}_{j=0}\mu ((\tilde x^j))\Gamma (f,f)(\tilde x^j).\end{align*}
This leads to 
\begin{align*}\int   f^{2}\X_{D}d\mu \leq  &\frac{N^2}{2\min \{x\in D:\mu(x)\}\delta}  \sum_{x\in D}\pi (x)\Gamma (f,f)(x) \\  =&\frac{N^2}{2\min \{x\in D:\mu(x)\}\delta} \mu(\Gamma (f,f)\X_D).\end{align*}
Gathering everything together gives
\begin{align*}Var_{\mu}(f)\leq\left(\frac{N^2}{2\min \{x\in D:\mu (x)\}\delta}+d_{1}\right)\int \Gamma(f,f) d\mu.  \end{align*}
\end{proof}

\section{proof of   Talagrand inequality for the invariant measure.}\label{proofTal}
Now we can prove concentration properties. At first we present the general proposition that connects the Poincar\'e inequality of Theorem \ref{invPoin}  with concentration of measure properties. The concentration properties will be based on the following proposition, that follows closely the approach in \cite{Led} (see also \cite{Led0}, \cite{Bo-Le}, \cite{A-M-S} and \cite{A-S}). We will also use elements from \cite{B-R} since one of the main conditions (\ref{conCon1}), will refer to the  bounded function $F_r=\min\{F,r\}$. 
\begin{proposition}\label{nnpropcon} Assume  that the Poincar\' e inequality\begin{align*}Var_{\mu}(f)\leq C_{0}\int \Gamma(f,f) d\mu
\end{align*}
holds, and    that for some  $\lambda$ such that $\lambda<\frac{1}{\sqrt{C_0C_3}}$
 \begin{align}\label{conCon1}\mu(\Gamma(e^{\lambda   F_r/2},e^{\lambda   F_r/2}))\leq  \lambda  ^2 C_3  \mu\left(  e^{\lambda(t) F_r}  \right).\end{align}
  Then the following concentration inequality holds   
\[\mu \left( \left\{F  >r \right\} \right)\leq  \lambda_0  e^{\lambda \mu (F)} e^{-\lambda  r}.\] for some $\lambda_0>0$.
Furthermore,
\[\mu \left(  \{P_tF-\mu f>r\} \right)\leq \lambda_0  e^{\lambda \mu (F)}e^{-\lambda  r}.\]
\end{proposition}
\begin{proof} 
From the Poincar\' e inequality, for $f=e^{\frac{\lambda F_r}{2}}$, we have 
  \[\mu (e^{\lambda F_r})\leq C_0 \mu(\Gamma(e^{\frac{\lambda}{2} F_r},e^{\frac{\lambda}{2} F_r}))+\left( \mu e^{\frac{\lambda}{2} F_r }\right)^2.\]
If we bound  the carr\'e du champ from condition (\ref{conCon1}) 
  \[\mu (e^{\lambda F_r})\leq C_0  \lambda  ^2 C_3  \mu\left(  e^{\lambda F_r}  \right)+\left( \mu e^{\frac{\lambda}{2} F_r }\right)^2.\]
 For $\lambda<\frac{1}{\sqrt{C_0C_3}}$ we get 
  \[\mu (e^{\lambda F_r})\leq \frac{1}{1-\lambda^2C_0C_3}\left( \mu e^{\frac{\lambda}{2} F_r }\right)^2.\]
 Iterating this gives
  \[\mu (e^{\lambda F_r})\leq \prod_{k=0}^{n-1}\left( \frac{1}{1-\frac{\lambda^2C_0C_3}{4^k}}\right)^{2^k}\left( \mu ( e^{\frac{\lambda}{2^n} F_r })\right)^{2^n}.\]
 We notice that  $\left( \mu( e^{\frac{\lambda}{2^n} F_r })\right)^{2^n}\rightarrow e^{\lambda \mu (F_r)}$ as $n\rightarrow \infty$ and that
 
  $\lambda_0:=\prod_{k=0}^{n-1}\left( \frac{1}{1-\frac{\lambda^2C_0C_3}{4^k}}\right)^{2^k}<\infty$ for $\lambda<\frac{1}{\sqrt{C_0C_3}}$. So we get 
  \[\mu (e^{\lambda F_r})\leq \lambda_0  e^{\lambda \mu (F_r)}<\infty  .\]
 Since 
\[\{P_t F_r<r\}=\{P_t F<r\}\] we can  apply Chebyshev's inequality
\[\mu \left(  \{P_tF>r\} \right)\leq e^{-\lambda r}\mu (e^{\lambda P_t F_r})\leq e^{-\lambda r}\mu (P_t e^{\lambda F_r})=e^{-\lambda r}\mu ( e^{\lambda F_r}).\]
  because of Jensen's inequality and  the invariant measure property $\mu P_t=\mu$.  Substitute $F$ with $F-\mu (F)$ and the result follows. 

Similarly, since   
\[\{  F_r<r\}=\{  F<r\}\] we also have 
\[\mu \left(  \{ F>r\} \right)\leq e^{-\lambda r}\mu (e^{\lambda   F_r}).\]

\end{proof}

 To complete the proofs of  concentration theorems \ref{thmCon2} and  \ref{talagTheo} and of Corollary \ref{thmCon2}, we need to verify  (\ref{conCon1}) . We start with  Theorem \ref{talagTheo}. We have to show condition (\ref{conCon1}) for $F(x)= \sum_{i=1}^Nx^i$. This will be the subject of the next   lemma.  \begin{Lemma}\label{ASbound2} Assume the PJMP as described in (\ref{phi2})-(\ref{eq:generator0}). Let  $F(x)= \sum_{i=1}^Nx^i$ for $x=(x^1,...,x^N)\in \R^N_+$. Then
 \[\mu(\Gamma(e^{\lambda F_r/2},e^{\lambda F_r/2}))\leq C_3 \lambda ^2   \mu\left(  e^{\lambda F_r}\right)\]
where $F_r=\min(F(x),r)$ for $r>0$.
\end{Lemma}
\begin{proof} From the definition of the carr\'e du champ
\[\mu(\Gamma(e^{\lambda F_r/2},e^{\lambda F_r/2}))= \sum_i\mu\left(\underbrace{\phi(x^i)  (e^{\lambda F_r(x)/2}-e^{\lambda F_r(\Delta_i(x))/2}) ^2}_{M_i}\right). \]
To bound $\mu (M_i)$ we will  distinguish four cases:

a) Consider the set $A:=\left\{ x: F(x)\geq r \  \text{  and } \   F(\Delta_i(x))\geq r\right\}$. Then, for $x \in A$ $F_r(\Delta_i(x))=F_r(x)=r$ and so $\mu (M_i\X_A)=0$.

b) Consider the set $B:=\left\{x:  F(x)\geq r \  \text{  and } \   F(\Delta_i(x))\leq r\right\}$. Then, for $x \in B$,
\[F_r(\Delta_i(x))=\sum_{j,j\neq i}\Delta_i(x)^j<r=F_r(x)\leq \sum_j x^j\]
so  that 
 \begin{align*} \mu (M_i\X_B)&\leq \lambda^2 \mu\left(\phi(x^i)  e^{\lambda F_r(x)}(  F_r(x)- F_r(\Delta_i(x)) ) ^2\right)\leq \\ & \leq \lambda^2 e^{\lambda r} \mu\left(\phi(x^i) (  F_r(x)- F_r(\Delta_i(x)) ) ^2\right).\end{align*}
 Since $F_r(\Delta_i(x))=\sum_{j,j\neq i}W_{i \to j}+\sum_{j,j\neq i}x^j<r\leq \sum_j x^j$ we have 
\[ F_r(x)- F_r(\Delta_i(x))=r-(\sum_{j,j\neq i}x^j+\sum_{j,j\neq i}W_{i \to j})<x^i-\sum_{j,j\neq i}W_{i \to j}\]
which leads to
\begin{align*}  \mu (M_i\X_B)\leq \lambda^2 e^{\lambda r} \mu\left(\phi(x^i) ( x^i-\sum_{j,j\neq i}W_{i \to j} ) ^2\right)=C^i\lambda^2 e^{\lambda r}=C^i\lambda^2 \mu (e^{\lambda F_r} \X_B)\end{align*}
where above we denoted $C^i=   \mu (\phi(x^i) ( x^i)^2) +N_0^2 \mu (\phi(x^i) )$ and computed 
$$e^{\lambda F_r} \X_B=e^{\lambda r} \X_B=\mu(e^{\lambda r} \X_B)=\mu(e^{\lambda F_r} \X_B).$$

c) Consider the set $C:=\left\{   F(\Delta_i(x))\leq F(x)< r  \right\}$. Then, for $x\in C$,
\[F_r(x)- F_r(\Delta_i(x))=\sum_j x^j -(\sum_{j,j\neq i}x^j+\sum_{j,j\neq i}W_{i \to j})=x^i-\sum_{j,j\neq i}W_{i \to j}\geq 0,\]
so that 
\begin{align*} \mu (M_i\X_{C})\leq &\lambda^2 \mu\left(\phi(x^i)  e^{\lambda F_r(x)}(  F_r(x)- F_r(\Delta_i(x)) ) ^2\right)
\\  \leq &  \lambda^2 \mu\left(\phi(x^i)  e^{\lambda F_r(x)}( x^i-\sum_{j,j\neq i}W_{i \to j} ) ^2\right)\\  \leq &  \lambda^2 \mu\left(\phi(x^i)  e^{\lambda F_r(x)}( x^i ) ^2\right)  .\end{align*}
Since $F_r\leq r$ we know that $\mu(e^{\lambda F_r })\leq e^{\lambda r}<\infty $  and so we can bound
\[ \mu (M_i\X_{C})\leq    
 \lambda^2  (\sup_{g: \mu(g)=1}\{\phi(x^i)(x^i)^2 g\})\mu\left(  e^{\lambda F_r } \X_{C} \right)\leq   \lambda^2 \vert \vert \vert \phi(x^i)(x^i)^2 \vert \vert \vert_{\infty}\mu\left(  e^{\lambda F_r } \X_{C} \right)\]

where
\[\vert\vert\vert f\vert\vert\vert_{\infty}=\sup_{g: \mu(g)=1}\{\mu(fg)\}.\]

d) Consider the set $D:=\left\{ F(x)< r \  \text{ and } \   F(x)< F(\Delta_i(x)) \right\}$. Then, for $x\in D$, 
\[ \sum_{j}x^{j}=F_r(x)< F_r(\Delta_i(x))\leq \sum_{j,j\neq i}W_{i \to j}+\sum_{j,j\neq i}x^j=F_r(x)+(\sum_{j,j\neq i}W_{i \to j}-x^{i})\]which means that $x^i$ is bounded by
\[x^{i}\leq  \sum_{j,j\neq i}W_{i \to j}\leq N_0 \]
and that 
$$0\leq F_r(\Delta_i(x))- F_r(x)\leq \sum_{j,j\neq i}W_{i \to j}-x^{i}.$$So, we can compute 
\begin{align*} \mu (M_i\X_D)\leq &\lambda^2 \mu\left(\phi(x^i)  e^{\lambda F_r(x)}(  F_r(x)- F_r(\Delta_i(x)) ) ^2\right)
 \\  \leq &   \lambda^2 \mu\left(\phi(x^i)  e^{\lambda F_r(\Delta_i(x))}( x^i-\sum_{j,j\neq i}W_{i \to j} ) ^2\right) \\  \leq &  
 \lambda^2  \mu\left(\phi(x^i)  e^{\lambda F_r(x)}e^{\lambda(\sum_{j,j\neq i}W_{i \to j}-x^{i})}( x^i-\sum_{j,j\neq i}W_{i \to j} ) ^2\right) \\   \leq &  
 N_0^2\phi( N_0)\lambda^2  e^{\lambda N_0 }\mu\left(  e^{\lambda F_r}  \X_D \right).\end{align*}
If we gather all four cases together, we finally obtain 
\[\mu(\Gamma(e^{\lambda F_R/2},e^{\lambda F_r/2}))\leq C \lambda ^2   \mu\left(  e^{\lambda F_r}  \right)\]
 for a constant
 \[C_3 =\max\{   \mu (\phi(x^i) ( x^i)^2) +N_0^2 \mu (\phi(x^i) ), \vert \vert \vert \phi(x^i)(x^i)^2 \vert \vert \vert_{\infty} , N_0^2\phi( N_0)  e^{\lambda N_0 } \}.  \]
\end{proof}

 In the remaining of the section we  prove the main concentration properties of the paper, presented in Theorem \ref{thmCon2} and Corollary \ref{thmCon1} . What remains is to present conditions so that  (\ref{conCon1}) of  Proposition \ref{nnpropcon} holds.  
 
As  one can see in the main tool to show concentration properties presented in  Proposition \ref{nnpropcon}, we need to bound  
$\mu(\Gamma(e^{\lambda f/2},e^{\lambda f/2}))$. In the case of diffusions, where $\mu (\Gamma(f,f))=\mu(\vert \vert \nabla f \vert \vert^2)$, for any smooth function $\psi$ one has 
\[\mu(\Gamma(\psi (f),\psi(f)))\leq \vert \vert \nabla f \vert \vert_{\infty}^2\mu(\psi'(f)^2)\] and so  one can bound $\mu(\Gamma(e^{\lambda f/2},e^{\lambda f/2}))\leq \frac{ \lambda^2}{4} \vert \vert \nabla f\vert \vert ^2_{\infty} \mu(e^{\lambda f}) $, and so the condition follows for functions $f$ such that  $\vert \vert \nabla f\vert \vert ^2_{\infty}<1$ (see \cite{Led} and \cite{Led0}).  In the case, as is in the current paper, of an energy expressed through differences,  where the chain rule is not satisfied, this cannot hold. However,  as demonstrated in \cite{A-S}    (see also  \cite{G-R} for applications), in the special situation where the semigroup is symmetric, one can  have  an analogue result, that is  
\begin{align*}\mu(\Gamma(e^{\lambda f/2},e^{\lambda f/2})) \leq \frac{ \lambda^2}{4}\vert\vert\vert \vert f \vert\vert\vert\vert_{\infty}^2\mu( e^{\lambda   f( x)}) \end{align*}
where now $\vert\vert\vert\vert f\vert\vert\vert\vert_{\infty}$ can be considered as a generalised norm of the gradient (see also \cite{Led}), given by the following expression
\[\vert\vert\vert\vert f\vert\vert\vert\vert_{\infty}=\sup\left\{\mathcal{E}(gf,f)-\frac{1}{2}\mathcal{E}(g,f^2); g:\vert\vert g\vert\vert_1\leq 1 \right\}\]
where $\mathcal{E}(f,g):=\lim_{t\rightarrow 0}\frac{1}{2t}\int \int(f(x)-f(y))^2p_t(x,dy)\mu (dx)$. Then, of course, for the concentration property to hold, one needs  functions that satisfy the  following    condition   $\vert\vert\vert\vert   f\vert\vert\vert\vert_{\infty}<1.$   
In our case  however,  we can still obtain the desired property   for a different class of functions, that satisfy   
\[\vert \vert \vert \phi(x^i)   D(f)^2\vert \vert \vert_{\infty}< 1\text{\ \ and \  \ }\vert \vert \vert \phi(x^i) e^{\lambda D(f)} D(f)^2\vert \vert \vert_{\infty}< 1\] where $\vert\vert\vert f\vert\vert\vert_{\infty}=\sup_{g: \mu(g)=1}\{\mu(fg)\})$.

 In the following lemma we show condition (\ref{conCon1}) under the hypothesis
 
  $\vert \vert \vert \phi(x^i)   D(f)^2\vert \vert \vert_{\infty}< 1$  and $\vert \vert \vert \phi(x^i) e^{\lambda D(f)} D(f)^2\vert \vert \vert_{\infty}< 1$ of Theorem \ref{thmCon2}, for  non-compact neurons   as in (\ref{eq:generator0})-(\ref{phi2}).\begin{Lemma}\label{ASbound4}Assume the PJMP as  described in (\ref{phi2})-(\ref{eq:generator0}). Assume functions $f$ such that 
  \[\vert \vert \vert \phi(x^i)   D(f)^2\vert \vert \vert_{\infty}< 1\text{\ \ and \  \ }\vert \vert \vert \phi(x^i) e^{\lambda D(f)} D(f)^2\vert \vert \vert_{\infty}< 1\] 
   Then
 \[\mu(\Gamma(e^{\lambda f_r/2},e^{\lambda f_r/2}))\leq C_3 \lambda ^2    \mu\left(  e^{\lambda f_r}  \right).\]
\end{Lemma}
\begin{proof}
From the definition of the carr\'e  du champ we compute \[\mu(\Gamma(e^{\lambda f_r/2},e^{\lambda f_r/2}))= \sum_{i=1}^N \mu\left(\underbrace{\phi(x^i)  (e^{\lambda f_r(x)/2}-e^{\lambda f_r(\Delta_i(x))/2}) ^2}_{M_i}\right). \]
a)Consider the set $A:=\left\{ x: f(x)\geq r \  \text{  and } \   f(\Delta_i(x))\geq r\right\}$. Then, for $x\in A$, $f_r(\Delta_i(x))=f_r(x)=r$ and so $\mu (M_i\X_A)=0$.

b) Consider the set $B:=\left\{x:  f(x)\geq r \  \text{  and } \   f(\Delta_i(x))\leq r\right\}$. Then, for $x \in B$,
\begin{align*} \mu (M_i\X_B)\leq \lambda^2 \mu\left(\phi(x^i)  e^{\lambda f_r(x)}D(f) ^2\right)\leq \lambda^2 e^{\lambda r} \mu\left(\phi(x^i) D(f ) ^2\right)\end{align*}
which leads to
\begin{align*}  \mu (M_i\X_B)\leq \mu\left(\phi(x^i) D(f ) ^2\right)\lambda^2 \mu (e^{\lambda f_r} \X_B)\end{align*}
since 
$$e^{\lambda f_r} \X_B=e^{\lambda r} \X_B=\mu(e^{\lambda r} \X_B)=\mu(e^{\lambda f_r} \X_B).$$

c) Consider the set $C:=\left\{   f(\Delta_i(x))\leq f(x)< r  \right\}$. 
Then, for $x\in C$,
\begin{align*} \mu (M_i\X_{C})\leq &\lambda^2 \mu\left(\phi(x^i)  e^{\lambda f_r(x)}(  f_r(x)- f_r(\Delta_i(x)) ) ^2\right)
\\  \leq &  \lambda^2 \mu\left(\phi(x^i)  e^{\lambda f_r(x)}D(f) ^2\right). \end{align*}
Since $f_r\leq r$ we know that $\mu(e^{\lambda f_r })\leq e^{\lambda r}<\infty $  and so  we can bound
\[ \mu (M_i\X_{C})\leq    
 \lambda^2  (\sup_{g: \mu(g)=1}\{\phi(x^i)D(f)^2 g\})\mu\left(  e^{\lambda f_r } \X_{C} \right)\leq   \lambda^2 \vert \vert \vert \phi(x^i)D(f)^2 \vert \vert \vert_{\infty}\mu\left(  e^{\lambda f_r } \X_{C} \right)\]

where
\[\vert\vert\vert f\vert\vert\vert_{\infty}=\sup   \left\{\mu(fg); \  g: \mu(g)\leq 1\right\} .\]

d) Consider the set $D:=\left\{ f(x)< r \  \text{ and } \   f(x)< f(\Delta_i(x)) \right\}$. Then, for $x\in D$,
\begin{align*} \mu (M_i\X_D)\leq &\lambda^2 \mu\left(\phi(x^i)  e^{\lambda f_r(\Delta_i(x))}(  f_r(x)- f_r(\Delta_i(x)) ) ^2\right)
 \\  \leq &  
 \lambda^2  \mu\left(\phi(x^i)  e^{\lambda f_r(x)}e^{\lambda D(f)}D(f ) ^2\right) \\   \leq &  
 \lambda^2  \vert \vert \vert \phi(x^i) e^{\lambda D(f)} D(f)^2\vert \vert \vert_{\infty}\mu\left(  e^{\lambda f_r}  \X_D \right).\end{align*}
where again we used that $e^{\lambda f_r} \X_D=e^{\lambda r} \X_D=\mu(e^{\lambda r} \X_D)=\mu(e^{\lambda f_r} \X_D).$

If we gather everything  together we finally obtain
\begin{align*}\mu(\Gamma(e^{\lambda f_r/2},e^{\lambda f_r/2}))\leq  &\lambda ^2\sum_{i=1}^N \left(2 \vert \vert \vert \phi(x^i)D(f)^2 \vert \vert \vert_{\infty}+\vert \vert \vert \phi(x^i) e^{\lambda D(f)} D(f)^2\vert \vert \vert_{\infty}\right)\mu\left(  e^{\lambda f_r}  \right)\\ \leq & 3N\lambda ^2   \mu\left(  e^{\lambda f_r}  \right).\end{align*}
and the lemma  follows for some constant $C_3=3N$.
\end{proof}

\end{document}